\documentclass[a4paper,twoside,11pt]{article}

\usepackage{a4wide, fancyhdr, amsmath, amssymb, mathtools, yfonts}
\usepackage{mathrsfs}
\usepackage{graphicx}
\usepackage{tikz}
\usepackage[all]{xy}
\usepackage[utf8]{inputenc}
\usepackage{amsthm}
\usepackage[english]{babel}
\usepackage{chngcntr}
\usepackage{ifthen}
\usepackage{calc}
\usepackage{hyperref}
\usepackage{authblk}
\usepackage{tikz-cd}
\numberwithin{equation}{section}

\newcommand{\Z}{\mathbb{Z}}
\newcommand{\Q}{\mathbb{Q}}

\newcommand{\A}{\mathbb{A}}

\newcommand\Gal{\mathrm{Gal}}

\usepackage[OT2,T1]{fontenc}
\DeclareSymbolFont{cyrletters}{OT2}{wncyr}{m}{n}
\DeclareMathSymbol{\Sha}{\mathalpha}{cyrletters}{"58}

\newtheorem{lemma}{Lemma}[section]
\newtheorem{theorem}[lemma]{Theorem}

\newtheorem*{remark}{Remark}

\title{\vspace{-\baselineskip}\sffamily\bfseries A note on the Hasse norm principle}
\author[1]{Peter Koymans\thanks{Institute for Theoretical Studies, ETH Z\"urich, 8006 Z\"urich, Switzerland, peter.koymans@eth-its.ethz.ch}}
\author[2]{Nick Rome\thanks{Graz University of Technology, Institute of Analysis and Number Theory,
Kopernikusgasse 24/II, 8010 Graz, Austria, rome@tugraz.at}}
\affil[1]{ETH Z\"urich}
\affil[2]{TU Graz}
\date{\today}

\begin{document}
\maketitle

\begin{abstract}
Let $A$ be a finite, abelian group. We show that the density of $A$-extensions satisfying the Hasse norm principle exists, when the extensions are ordered by discriminant. This strengthens earlier work of Frei--Loughran--Newton \cite{FLN}, who obtained a density result under the additional assumption that $A/A[\ell]$ is cyclic with $\ell$ denoting the smallest prime divisor of $\# A$.
\end{abstract}

\section{Introduction}
The Hasse norm theorem states that if $K/k$ is a cyclic extension of number fields, then $c \in k^{\times}$ is a \emph{global norm} if and only if it is a \emph{local norm} everywhere. In other words, $N_{K/k} K^{\times} = k^{\times} \cap N_{K/k} \A_K^{\times}$, where $N_{K/k}$ denotes the norm map and $\A_K$ the adeles. Unfortunately, this property does not hold in general for non-cyclic extensions (see \cite[Exercise 5]{CF} for a nice example). If, for an extension $K/k$, it does hold that $N_{K/k} k^{\times} = k^{\times} \cap N_{K/k} \A_K^{\times}$, we say that the extension satisfies the \emph{Hasse norm principle} (or simply HNP). Recently, Frei--Loughran--Newton~\cite{FLN} have shown that for any finite, abelian, non-cyclic group $A$ and any number field $k$, there exists an extension $K/k$ with Galois group $A$ such that the Hasse norm principle fails. They also demonstrated \cite[Theorem 1.1]{FLN} that these failures are rare in the following sense: writing $\Delta(K/k)$ for the relative discriminant and writing $\ell$ for the smallest prime dividing $\# A$, then if $A/A[\ell]$ is cyclic
\begin{align}
\label{eFLNcyclic}
\lim_{X \rightarrow \infty} \frac{\#\{K/k: \text{Gal}(K/k) \cong A, N_{k/\Q}(\Delta(K/k)) \leq X \text{ and } K/k \text{ fails the HNP}\}}{\#\{K/k: \text{Gal}(K/k) \cong A, N_{k/\Q}(\Delta(K/k)) \leq X\}} = 0.
\end{align}
The purpose of this note is to establish, by an alternative method, the following strengthenings of their result.

\begin{theorem}
\label{tHNPLimit}
Let $k$ be a number field. Let $A$ be a non-trivial, finite, abelian group. Then the limit
\begin{align}
\label{eHNPLimit}
\lim_{X \rightarrow \infty} \frac{\#\{K/k: \Gal(K/k) \cong A, N_{k/\Q}(\Delta(K/k)) \leq X \textup{ and } K/k \textup{ satisfies the HNP}\}}{\#\{K/k: \Gal(K/k) \cong A, N_{k/\Q}(\Delta(K/k)) \leq X\}}
\end{align}
exists.
\end{theorem}

Moreover, we are able to locate the value of this limit, depending on the structure of the group $A$.

\begin{theorem}
\label{tLimit1}
Let $k$ be a number field. Let $A$ be a non-trivial, finite, abelian group. Let $\ell$ be the smallest prime divisor of $\# A$. If $A/A[\ell]$ is cyclic, then the limit (\ref{eHNPLimit}) equals $1$.
\end{theorem}

\begin{theorem}
\label{tLimit2}
Let $k$ be a number field. Let $A$ be a non-trivial, finite, abelian group. Let $\ell$ be the smallest prime divisor of $\# A$. If $A/A[\ell]$ is not cyclic, then the limit (\ref{eHNPLimit}) lies in the open interval $(0, 1)$.
\end{theorem}

Let $A$ be an abelian group such that $A/A[\ell]$ is not cyclic. Previously, Frei--Loughran--Newton \cite[Theorem 1.4 \& Theorem 1.6]{FLN} have shown that
\[
\liminf_{X \rightarrow \infty} \frac{\#\{ K/k: \text{Gal}(K/k) \cong A, N_{k/\Q}(\Delta(K/k)) \leq X \text{ and } K/k \text{ satisfies the HNP}\}}{\#\{ K/k: \text{Gal}(K/k) \cong A, N_{k/\Q}(\Delta(K/k)) \leq X\}} > 0
\]
and
\[
\limsup_{X \rightarrow \infty} \frac{\#\{ K/k: \text{Gal}(K/k) \cong A, N_{k/\Q}(\Delta(K/k)) \leq X \text{ and } K/k \text{ satisfies the HNP}\}}{\#\{ K/k: \text{Gal}(K/k) \cong A, N_{k/\Q}(\Delta(K/k)) \leq X\}} < 1.
\]
Theorem \ref{tLimit2} recovers and strengthens these results, while Theorem \ref{tLimit1} recovers equation (\ref{eFLNcyclic}). The key advantage of Theorems \ref{tHNPLimit} and \ref{tLimit2} is that they concern a genuine limit for all $A$, as opposed to a limsup or liminf.

Theorem \ref{tLimit2} tells us how often the Hasse norm principle fails for groups $A$ such that $A/A[\ell]$ is not cyclic. However, in the case that this quotient is cyclic, these density results give no indication as to precisely how many Hasse norm principle failures there can be. Asymptotic formulae counting the number of Hasse norm principle failures have been established in the case that $A \cong \Z/2\Z \times \Z/2\Z$ in \cite{biquad}, or if $A \cong \prod_{i=1}^r \left(\Z/p_i \Z\right)^{n_i}$ for distinct primes $p_i$ and natural numbers $n_i$ such that $n_1 \cdots n_r \leq 2$ in \cite[Corollary 1.3]{WA}. Finally, we note that if one orders by conductor instead of discriminant, then the corresponding limit exists and is always equal to 1. This result was established by Frei--Loughran--Newton in \cite[Theorem 1.9]{FLN2}.

\begin{remark}
Our aim with this short paper is the strengthenings of \cite{FLN} stated above. Our approach, though, demonstrates an alternative method by which one can prove strong density results concerning counting number fields with prescribed local conditions. Our method has the advantage of being relatively quick and easy to apply, but the harmonic analysis approach of \cite{FLN} and \cite{FLN2} might be more suitable in certain settings. 
\end{remark}

\subsection*{Acknowledgements}
The authors are grateful for the hospitality of Connor O'Neil during part of the process of writing this paper. PK acknowledges the support of Dr. Max R\"ossler, the Walter Haefner Foundation and the ETH Z\"urich Foundation. NR is supported by the Austrian Science Fund (FWF) project ESP 441-NBL. This work was completed while the authors were based at the University of Michigan.

\subsection*{Notation}
Henceforth, $k$ will be a fixed number field and $A$ a fixed abelian group. We endow all our abelian groups with the discrete topology, and all homomorphisms are tacitly assumed to be continuous. For any field $K$, we will denote by $G_K$ the absolute Galois group $\text{Gal}(\overline{K}/K)$ for some fixed choice of algebraic closure $\overline{K}$. We will write $\Omega_k$ for the set of places of $k$. 

An $A$-extension $K/k$ is by definition an epimorphism $\varphi: G_k \rightarrow A$. In particular, given such a $\varphi$, it makes sense to speak of the restriction homomorphism to a place $v \in \Omega_k$, namely $\text{res}_v(\varphi): G_{k_v} \rightarrow A$, and the associated decomposition group $D_v \subseteq A$, which is by definition the image of $\text{res}_v(\varphi)$. We write $D(\varphi)$ for the absolute norm of the discriminant of $\varphi$.

\section{Ingredients}
We have the following criterion for the Hasse norm principle.

\begin{lemma}
\label{lHNP}
Let $K/k$ be an $A$-extension. Then the Hasse norm principle holds if and only if
\[
\textup{Hom}(\wedge^2(A), \Q/\Z) \rightarrow \prod_v \textup{Hom}(\wedge^2(D_v), \Q/\Z)
\]
is injective.
\end{lemma}

\begin{proof}
See \cite[Section 6]{FLN}.
\end{proof}

In order to understand the denominator of the ratio appearing in our limit, we will appeal to Wright's theorem on abelian discriminants.

\begin{theorem}[{\cite[Theorem 1.2]{Wright}}]
\label{thm:Wright}
Let $k$ be a number field and $A$ a non-trivial, finite, abelian group. Let $\ell$ denote the smallest prime dividing $\# A$. Then there exists a positive constant $c(k, A)$ such that
\[
\#\{\varphi \ A\textup{-extension}: D(\varphi) \leq X\} \sim c(k, A) X^{\frac{\ell}{\# A \cdot (\ell - 1)}} (\log X)^{-1 + \frac{\#A[\ell] - 1}{[k(\zeta_\ell) : k]}},
\] 
as $X \rightarrow \infty$.
\end{theorem}

The final major ingredient in our proof is the following rather special case of a theorem of Wood, which allows one to count abelian extensions with local conditions imposed at a finite number of places. We start by building a formal model as in \cite[Subsection 2.2]{Wood}. Define $\Omega := \prod_{v \in \Omega_k} \text{Hom}(G_{k_v}, A)$. Given a finite set of places $S$ and a local specification $\Sigma := (\Sigma_v)_{v \in S}$ on $S$ with $\Sigma_v \in \text{Hom}(G_{k_v}, A)$, we define $\text{Box}(\Sigma)$ to be the subset of $\Omega$ projecting to $\Sigma_v$ for all $v \in S$. We uniquely specify a probability measure $\mathbb{P}$ on the algebra of subsets generated by sets of the shape $\text{Box}(\Sigma_v)$ by demanding that
\begin{itemize}
\item $\frac{\mathbb{P}(\text{Box}(\Sigma_v))}{\mathbb{P}(\text{Box}(\Sigma_v'))} = \frac{N_{k/\Q}(v)^{\mathbf{1}_{\Sigma_v' \textup{ ramifies}}}}{N_{k/\Q}(v)^{\mathbf{1}_{\Sigma_v \textup{ ramifies}}}}$ for all $\Sigma_v, \Sigma_v' \in \text{Hom}(G_{k_v}, A)$ and
\item $\text{Box}(\Sigma_{v_1}), \dots, \text{Box}(\Sigma_{v_s})$ are independent for all distinct places $v_1, \dots, v_s$.
\end{itemize}
We adopt the convention here that the norm of an archimedean place is 1.

We say that a local specification $\Sigma$ is viable if there exists at least one $A$-extension $\varphi: G_k \rightarrow A$ restricting to $\Sigma_v$ for all $v \in S$. Let $T$ be all the places of $k$ that lie above $2$. From now on we will assume that $T \subseteq S$. Then, by Grunwald--Wang, there exists a finite list $\Sigma(1), \dots, \Sigma(k)$ of local specifications on $T$ such that $\Sigma$ is viable if and only if $\Sigma$ restricts to $\Sigma(i)$ for some $i$. We write $\Omega_{\text{GW}}$ for the subspace of $\Omega$ projecting to some $\Sigma(i)$ and we write $C(K/k)$ for the product of the primes in $k$ that ramify in the extension $K/k$. We define the empirical probability
\[
\textup{Pr}(\Sigma) := \lim_{X \rightarrow \infty} \frac{\#\{\varphi \ A\textup{-extension}: N_{k/\Q}(C(K/k)) \leq X \textup{ and } \textup{res}_v(\varphi) = \Sigma_v \forall v \in S\}}{\#\{\varphi \ A\textup{-extension}: N_{k/\Q}(C(K/k)) \leq X\}}.
\]

\begin{theorem}[{\cite[Theorem 2.1]{Wood}}]
\label{thm:Wood}
Let $k$ be a number field and let $A$ be a non-trivial, finite, abelian group. Let $S$ be a set of places of $k$ including all places above $2$. Suppose that $\Sigma = (\Sigma_v)_{v \in S}$ is a viable local specification. Then $\textup{Pr}(\Sigma)$ exists and
\[
\textup{Pr}(\Sigma) = \mathbb{P}(\textup{Box}(\Sigma) | \Omega_{\textup{GW}}).
\]
\end{theorem}

\section{\texorpdfstring{Proof of Theorem \ref{tHNPLimit}}{Proof that limit exists}}
Let $k$ be a number field and let $A$ be a non-trivial, finite, abelian group. Write $\ell$ for the smallest prime divisor of $\# A$, write $B := A/A[\ell]$ and write $\pi: A \rightarrow B$ for the natural quotient map. We list the $B$-extensions of $k$ as $\varphi_1, \varphi_2, \dots$ ordered by their absolute discriminant (picking an arbitrary ordering for extensions with equal discriminant), where each $\varphi_i$ is an epimorphism from $G_k$ to $B$. We define the quantity
\[
N_i(X) := \frac{\#\{\varphi \ A\text{-extension} : D(\varphi) \leq X, \pi \circ \varphi = \varphi_i, \varphi \text{ satisfies HNP}\}}{\#\{\varphi \ A\text{-extension} : D(\varphi) \leq X\}}.
\]
The proof of Theorem \ref{tHNPLimit} proceeds in two steps. First, we will apply the dominated convergence theorem to fix the $B$-extension. Second, we will use Theorem \ref{thm:Wood} to calculate the limit for a fixed $B$-extension. The key point is that we are able to reduce the problem from $A$-extensions to $A[\ell]$-extensions by this method. Since the discriminant is a fair counting function for $A[\ell]$-extensions, Wood's result applies.

\subsection*{Step 1: Dominance}
Our goal will be to apply dominated convergence, so we aim to find an appropriate upper bound for $N_i(X)$. As a first step, we drop the condition that $\varphi$ satisfies the Hasse norm principle, i.e. we define 
\[
N_i'(X) := \frac{\#\{\varphi \ A\text{-extension} : D(\varphi) \leq X, \pi \circ \varphi = \varphi_i\}}{\#\{\varphi \ A\text{-extension} : D(\varphi) \leq X\}}.
\]

\begin{lemma}
\label{lUpper}
There exists $\delta > 0$ depending only on $k$ and $A$ such that
\[
N_i'(X) \ll_{k, A} \frac{\ell^{[k : \Q] \cdot \omega(D(\varphi_i))}}{\prod_{p \mid D(\varphi_i)} p^{1 + \delta}},
\]
where the implied constant may depend on $k$ and $A$, but not on $i$.
\end{lemma}

\begin{remark}
One can view the numerator of $N_i'(X)$ as fixing an $A/A[\ell]$-extension and then twisting by $A[\ell]$-extensions to get a sequence of $A$-extensions. A key fact, which we will exploit, is that primes which ramify in the $A[\ell]$ part of the extension, but not in the $A/A[\ell]$ part, have the lowest ramification exponent and thus smallest weight in the discriminant.
\end{remark}

\begin{proof}
By a result of Shafarevich \cite{Sha}, there exists a set of places $S$ of $k$ such that for all finite places $v$ of $k$ not lying in $S$, there exists a surjective homomorphism $G_k \rightarrow \mathbb{F}_\ell$ ramified at $v$ and unramified at all places outside $S \cup \{v\}$. By enlarging $S$ if necessary, we may assume that $S$ contains all archimedean places and all places dividing $2 \cdot \# A$. If $\varphi_i: G_k \rightarrow B$ does not admit a lift to $A$, then $N_i'(X) = 0$. Hence, we now suppose that $\varphi_i$ does admit a lift to $A$. By Shafarevich's result, we may choose a lift $\widetilde{\varphi_i}: G_k \rightarrow A$ of $\varphi_i$ only ramified at $S$ and the ramification locus of $\varphi_i$. Furthermore, any other lift of $\varphi_i$ can be written uniquely as $\widetilde{\varphi_i} + t$ for some homomorphism $t: G_k \rightarrow A[\ell]$. 

Given an epimorphism $\varphi: G_k \rightarrow C$ for any finite, abelian group $C$, we write $\Delta(\varphi)$ for the discriminant ideal of $k$. In case $\varphi$ is not surjective, we interpret this as the discriminant of the corresponding \'etale algebra. With this notation set, $N_i'(X)$ becomes
\begin{align}
\label{eNiXrewrite}
N_i'(X) = \frac{\#\{t : G_k \rightarrow A[\ell] : \widetilde{\varphi_i} + t \textup{ surj.}, N_{k/\Q}(\Delta(\widetilde{\varphi_i} + t)) \leq X\}}{\#\{\varphi \ A\text{-extension} : D(\varphi) \leq X\}}.
\end{align}
If $v \in \Omega_k - S$ ramifies in $\varphi_i$, then we claim that
\begin{align}
\label{eDiscLower}
v(\Delta(\widetilde{\varphi_i} + t)) \geq (1 + \delta) \cdot \left(1 - \frac{1}{\ell}\right) \cdot \# A,
\end{align}
for some $\delta > 0$ depending only on $k$ and $A$. Write $e$ for the ramification exponent of $v$ in $\varphi_i$. If $\ell \mid e$, then its ramification exponent in $\widetilde{\varphi_i} + t$ is $\ell \cdot e$. If $\ell \nmid e$, then we have $e > \ell$, since $\ell$ is the smallest prime divisor of $\# A$. The claim therefore holds in both cases by the conductor-discriminant formula.
If $v \in \Omega_k - S$ ramifies in $t$ but not in $\varphi_i$, then we have
\begin{align}
\label{eDiscExact}
v(\Delta(\widetilde{\varphi_i} + t)) = \left(1 - \frac{1}{\ell}\right) \cdot \# A.
\end{align}
For a set of places $T$ of $\Omega_k$, we write $\Delta_T(\varphi)$ and $D_T(\varphi)$ for the largest divisor of respectively $\Delta(\varphi)$ and $D(\varphi)$ coprime to all the places in $T$. Define $S(i)$ to be the union of $S$ with the ramification locus of $\varphi_i$. Using equation (\ref{eDiscLower}), we upper bound the numerator of (\ref{eNiXrewrite}) as
\begin{align}
\label{eNiXUpper1}
\#\left\{t : G_k \rightarrow A[\ell] : \widetilde{\varphi_i} + t \textup{ surj.}, N_{k/\Q}(\Delta_{S(i)}(\widetilde{\varphi_i} + t)) \leq \frac{X}{\prod_{p \mid D_S(\varphi_i)} p^{(1 + \delta) \cdot \left(1 - \frac{1}{\ell}\right) \cdot \# A}}\right\}.
\end{align}
We claim that for all finite subsets $T$ of $\Omega_k$ containing $S$
\begin{multline}
\label{eTrick}
\#\{t : G_k \rightarrow A[\ell] : \widetilde{\varphi_i} + t \textup{ surj.}, N_{k/\Q}(\Delta_T(\widetilde{\varphi_i} + t)) \leq X\} \ll_{k, A} \\
\ell^{\# T} \cdot \#\{t : G_k \rightarrow A[\ell] : N_{k/\Q}(\Delta_S(\widetilde{\varphi_i} + t)) \leq X\}.
\end{multline}
Indeed, if $\Delta_T(\widetilde{\varphi_i} + t) \leq X$, then we may find $t'$ only ramified at $T$ such that $\widetilde{\varphi_i} + t + t'$ is unramified at $T - S$ and therefore $\Delta_T(\widetilde{\varphi_i} + t) = \Delta_S(\widetilde{\varphi_i} + t + t')$. Since there are $\ll_{k, A} \ell^{\# T}$ choices for $t'$, the claim follows. 

We now apply \eqref{eTrick} with $T$ equal to $S(i)$. Together with equation (\ref{eNiXUpper1}), this yields
\begin{align}
\label{eNiXUpper2}
N_i'(X) \ll_{k, A} \frac{\ell^{\# S(i)} \#\left\{t : G_k \rightarrow A[\ell] : N_{k/\Q}(\Delta_S(\widetilde{\varphi_i} + t)) \leq \frac{X}{\prod_{p \mid D_S(\varphi_i)} p^{(1 + \delta) \cdot \left(1 - \frac{1}{\ell}\right) \cdot \# A}}\right\}}{\#\{\varphi \ A\text{-extension} : D(\varphi) \leq X\}}.
\end{align}
For the numerator, note that $N_{k/\Q}(\Delta(t))^{\# B} \ll_{k, A} N_{k/\Q}(\Delta_S(t))^{\# B} \leq N_{k/\Q}(\Delta_S(\widetilde{\varphi_i} + t))$ by equations (\ref{eDiscLower}) and (\ref{eDiscExact}). Therefore, it remains to upper bound 
\[
\#\left\{t : G_k \rightarrow A[\ell] : N_{k/\Q}(\Delta(t))^{\# B} \leq \frac{X}{\prod_{p \mid D_S(\varphi_i)} p^{(1 + \delta) \cdot \left(1 - \frac{1}{\ell}\right) \cdot \# A}}\right\}. 
\]
For surjective maps $t$, this set corresponds precisely to $A[\ell]$-extensions and we apply Theorem \ref{thm:Wright}. If $t$ is not surjective, then the image is a strict subspace of $A[\ell]$ and therefore the exponent in Theorem \ref{thm:Wright} becomes smaller. We conclude that such $t$ will contribute a negligible amount to the ratio. In this way, we obtain the upper bound
\begin{multline*}
\#\left\{t : G_k \rightarrow A[\ell] : N_{k/\Q}(\Delta(t))^{\# B} \leq \frac{X}{\prod_{p \mid D_S(\varphi_i)} p^{(1 + \delta) \cdot \left(1 - \frac{1}{\ell}\right) \cdot \# A}}\right\} \ll_{k, A} \\
\frac{X^{\frac{\ell}{\#A \cdot (\ell - 1)}} (\log X)^{-1 + \frac{\#A[\ell] - 1}{[k(\zeta_\ell) : k]}}}{\prod_{p \mid D_S(\varphi_i)} p^{1 + \delta}}.
\end{multline*}
Plugging the upper bound back into equation (\ref{eNiXUpper2}) gives 
\[
N_i'(X) \ll_{k, A} \frac{\ell^{\# S(i)}}{\prod_{p \mid D_S(\varphi_i)} p^{1 + \delta}}, 
\]
where we applied Theorem \ref{thm:Wright} to lower bound the denominator of equation (\ref{eNiXUpper2}). Inserting the bounds $D_S(\varphi_i) \gg_{k, A} D(\varphi_i)$ and $\ell^{\# S(i)} \ll_{k, A} \ell^{[k : \Q] \cdot \omega(D(\varphi_i))}$ ends the proof of the lemma.
\end{proof}

By the discriminant multiplicity conjecture for abelian extensions \cite[Theorem 1.6]{KW} and Lemma \ref{lUpper}, it follows that 
\[
\sum_{i = 1}^\infty N_i(X) \leq \sum_{i = 1}^\infty N_i'(X) \ll_{k, A} \sum_{i = 1}^\infty \frac{\ell^{[k : \Q] \cdot \omega(D(\varphi_i))}}{\prod_{p \mid D(\varphi_i)} p^{1 + \delta}}
\]
converges. Therefore, by dominated convergence, it suffices to show that $\lim_{X \rightarrow \infty} N_i(X)$ exists for each fixed $i$.

\subsection*{Step 2: Convergence}
From now on, fix some integer $i \geq 1$. If $t: G_k \rightarrow A[\ell]$ is an epimorphism, then so is $\widetilde{\varphi_i} + t$ (we warn the reader that the converse however may fail). Since the set of non-surjective twists $t$ have zero density in the total set of twists (by the argument concerning the exponent of the logarithm in Theorem \ref{thm:Wright} given above), we may ignore these non-surjective twists. Choose a viable set of local conditions $\Sigma_v \in \text{Hom}(G_{k_v}, A[\ell])$ for each $v \in S(i)$, where $S(i)$ is defined in Step 1. We will instead determine the limit
\begin{equation}\label{eConvLimit}
\lim_{X \rightarrow \infty} \frac{\#\{\widetilde{\varphi_i} + t \ A\text{-extension} : D(\widetilde{\varphi_i} + t) \leq X, \text{res}_v(t) = \Sigma_v \ \forall v \in S(i), \widetilde{\varphi_i} + t \text{ satisfies HNP}\}}{\#\{\varphi \ A\text{-extension} : D(\varphi) \leq X\}}
\end{equation}
for each such choice of local conditions, where $\text{res}_v$ is the restriction map. The local conditions determine the $v$-adic valuation of $\Delta(\widetilde{\varphi_i} + t)$ for the places $v \in S(i)$. For $v \not \in S(i)$, we have that the $v$-adic valuation of $\Delta(\widetilde{\varphi_i} + t)$ equals the $v$-adic valuation of $\Delta(t)^{\# B}$ by equation (\ref{eDiscExact}). Instead consider the limit
\begin{align}
\label{eWoodLimit}
\lim_{X \rightarrow \infty} 
\frac{\#\{\widetilde{\varphi_i} + t \ A\text{-extension} : D(\widetilde{\varphi_i} + t) \leq X, \text{res}_v(t) = \Sigma_v \ \forall v \in S(i), \widetilde{\varphi_i} + t \text{ satisfies HNP}\}}{\#\{\widetilde{\varphi_i} + t \ A\text{-extension} : D(\widetilde{\varphi_i} + t) \leq X, \text{res}_v(t) = \Sigma_v \ \forall v \in S(i)\}}.
\end{align}
We claim that this limit is either $0$ or $1$. This will suffice to prove our theorem since the limit
\[
\lim_{X \rightarrow \infty} \frac{\#\{\widetilde{\varphi_i} + t \ A\text{-extension} : D(\widetilde{\varphi_i} + t) \leq X, \text{res}_v(t) = \Sigma_v \ \forall v \in S(i)\}}{\#\{\varphi \ A\text{-extension} : D(\varphi) \leq X\}}
\] 
exists by Theorem \ref{thm:Wood} applied to $A[\ell]$ and the base field $k$. Here we also use that the number of $A$-extensions and $A[\ell]$-extensions of $k$, ordered by discriminant, have the same order of magnitude thanks to Theorem \ref{thm:Wright}.

Let $\mathcal{C}$ be the collection of subgroups of $A$ that are generated by $2$ elements, of which at least one has order dividing $\ell$. Consider the homomorphism
\begin{align}
\label{eLocalCondition}
\text{Hom}(\wedge^2(A), \Q/\Z) \rightarrow \bigoplus_{v \in S(i)} \text{Hom}(\wedge^2(D_v), \Q/\Z) \oplus \bigoplus_{H \in \mathcal{C}} \text{Hom}(\wedge^2(H), \Q/\Z).
\end{align}
If this map is not injective, then we will show that the limit in equation (\ref{eWoodLimit}) is $0$. The key ingredient will be the following lemma based on local class field theory.

\begin{lemma}
Let $\varphi: G_k \rightarrow A$ be a surjective map lifting $\varphi_i$. If $v \not \in S(i)$, then $D_v \in \mathcal{C}$.
\end{lemma}

\begin{proof}
If $v \not \in S(i)$, then $v$ is tamely ramified in $\varphi$. The maximal abelian, tamely ramified extension of $k_v$ has Galois group isomorphic to $T \times \hat{\Z}$, where $T$ is a cyclic subgroup generated by tame inertia and $\hat{\Z}$ is generated by (a choice of) Frobenius. Therefore the map $\text{res}_v(\varphi): G_{k_v} \rightarrow A$ induces a map $T \times \hat{\Z} \rightarrow A$, which sends a generator of tame inertia to an element of order $\ell$, because $v \not \in S(i)$. By definition, $D_v$ is the image of this map, giving the lemma.
\end{proof}

If the map in \eqref{eLocalCondition} is not injective, then the map in Lemma \ref{lHNP} is also not injective for any $\widetilde{\varphi_i} + t$ in the numerator of \eqref{eWoodLimit} and thus the Hasse norm principle fails.

Now suppose that the map is injective. If for all $H \in \mathcal{C}$ there exists a $v$ such that $D_v = H$, then the map in Lemma \ref{lHNP} is also injective and thus the Hasse norm principle holds. Hence it suffices to prove that for each fixed $H \in \mathcal{C}$, we have
\[
\lim_{X \rightarrow \infty} \frac{\#\{\widetilde{\varphi_i} + t \ A\text{-extension} : D(\widetilde{\varphi_i} + t) \leq X, \text{res}_v(t) = \Sigma_v \, \forall v \in S(i), D_v \neq H \,\forall v \not \in S(i)\}}{\#\{\widetilde{\varphi_i} + t \ A\text{-extension} : D(\widetilde{\varphi_i} + t) \leq X, \text{res}_v(t) = \Sigma_v \, \forall v \in S(i)\}} = 0.
\] 
Since the cardinality $\#\text{Hom}(G_{k_v}, A[\ell])$ at tame places is bounded in terms of $A[\ell]$ only, the bound $\mathbb{P}(\text{Box}(\Sigma_v)) \gg_{k, A} N_{k/\Q}(v)^{-1}$ holds for all $\Sigma_v \in \text{Hom}(G_{k_v}, A[\ell])$ and all $v \not \in S(i)$. For some parameter $Y$, let $\widetilde{S}(i) = S(i) \cup \{v : N_{k/\Q}(v) \leq Y\}$. We aim to apply Wood's theorem to the group $A[\ell]$, in order to show that eventually, when one considers enough places, the probability that $D_v \neq H$ for all such places can be made arbitrarily small. However, it is possible that the $A/A[\ell]$-extension $\varphi_i$ is such that there exist places $v$ so that for all $A[\ell]$ twists $t$, we always have $D_v \neq H$. We avoid considering such places by picking some $a \in A$ such that $H = \langle a, b \rangle$ with $b$ of order $\ell$ and we will only consider places $v$ such that $\widetilde{\varphi_i}(\text{Frob}_v) = a$. By Theorem \ref{thm:Wood} applied to $A[\ell]$, we have that
\begin{align*}
\limsup_{X \rightarrow \infty} &
\frac{\#\{\widetilde{\varphi_i} + t \ A\text{-extension} : D(\widetilde{\varphi_i} + t) \leq X, \text{res}_v(t) = \Sigma_v \,\forall v \in S(i), D_v \neq H \,\forall v \in \widetilde{S}(i) \setminus S(i)\}}{\#\{\widetilde{\varphi_i} + t \ A\text{-extension} : D(\widetilde{\varphi_i} + t) \leq X, \text{res}_v(t) = \Sigma_v \, \forall v \in S(i)\}} \\
&\leq \prod_{\substack{v \in \widetilde{S}(i) \setminus S(i) \\ \widetilde{\varphi_i}(\text{Frob}_v) = a}} \left( 1 - \frac{1}{c \cdot N_{k/\Q}(v)} \right)
\end{align*} 
for some absolute constant $c > 0$ depending only on $k$ and $A$. The theorem then follows from a number field version of Mertens' theorem, see for instance \cite{Rosen}.

\section{\texorpdfstring{Proof of Theorem \ref{tLimit1}}{Value of limit}}
In this section we will show that the limit in equation (\ref{eHNPLimit}) equals $1$ in case $A/A[\ell]$ is cyclic.

Let $H$ be a subgroup of $A$ generated by two elements, say $H = \langle a, b \rangle$. We remind the reader that $\mathcal{C}$ is the collection of subgroups of $A$ that are generated by $2$ elements, of which at least one has order $\ell$. We claim that $H \in \mathcal{C}$. To show this, we decompose $A$ as
$$
A = (\Z/\ell\Z)^r \oplus \Z/k\Z
$$
for some $r \geq 0$ and some $k$ divisible by $\ell$. Write $a = (v_1, c_1)$ and $b = (v_2, c_2)$. By using the Euclidean algorithm on the last coordinate, we see that $\langle a, b \rangle = \langle (v_1', c_1'), (v_2', 0) \rangle$ and therefore $H \in \mathcal{C}$ as claimed. From the claim, we deduce that the map (\ref{eLocalCondition}) is always injective and therefore that the limit equals $1$.

\section{\texorpdfstring{Proof of Theorem \ref{tLimit2}}{Value of limit II}}
We will now prove Theorem \ref{tLimit2}. To this end, suppose that $A/A[\ell]$ is not cyclic. It follows from \cite[Corollary 1.10]{FLN2} that there exists at least one $A$-extension $\varphi$ satisfying the Hasse norm principle. Let $i$ be the unique integer such that $\pi \circ \varphi = \varphi_i$. By construction there exists a finite subset $T \subseteq \Omega_k$ and a choice of local specifications $(\Sigma_v)_{v \in T}$ such that for all sufficiently large real numbers $X$
\[
\{\widetilde{\varphi_i} + t \ A\text{-extension} : D(\widetilde{\varphi_i} + t) \leq X, \text{res}_v(t) = \Sigma_v \ \forall v \in T, \widetilde{\varphi_i} + t \text{ satisfies HNP}\}
\]
is non-empty. This forces the map in equation (\ref{eLocalCondition}) to be injective (since otherwise the above set would be empty), which in turn forces
\[
\lim_{X \rightarrow \infty} \frac{\#\{\widetilde{\varphi_i} + t \ A\text{-extension} : D(\widetilde{\varphi_i} + t) \leq X, \text{res}_v(t) = \Sigma_v \ \forall v \in T, \widetilde{\varphi_i} + t \text{ satisfies HNP}\}}{\#\{\widetilde{\varphi_i} + t \ A\text{-extension} : D(\widetilde{\varphi_i} + t) \leq X, \text{res}_v(t) = \Sigma_v \ \forall v \in T\}} = 1
\]
by the arguments in the previous section. We deduce that the limit in the main theorem is strictly greater than $0$.

It remains to show that the limit in the main theorem is strictly smaller than $1$. Take $\varphi$ to be an $A$-extension with all decomposition groups cyclic. Such an extension exists by \cite[Proposition 5.5]{FLN}. Let $i$ be the unique integer such that $\pi \circ \varphi = \varphi_i$. Take $\Sigma_v$ to be the totally split local conditions, which is viable. In this case, the map in equation (\ref{eLocalCondition}) becomes
\begin{equation}
\label{eNotInjective}
\text{Hom}(\wedge^2(A), \Q/\Z) \rightarrow \bigoplus_{H \in \mathcal{C}} \text{Hom}(\wedge^2(H), \Q/\Z)
\end{equation}
and it suffices to show that this map is not injective. 

Using that $A/A[\ell]$ is not cyclic, we have $\wedge^2(A/A[\ell]) \neq 0$. Hence there exists a non-trivial alternating pairing $f: A/A[\ell] \times A/A[\ell] \rightarrow \Q/\Z$, which induces a non-trivial alternating pairing $\tilde{f}: A \times A \rightarrow \Q/\Z$ via the quotient map. By construction $\tilde{f}$ is identically zero when restricted to $H$ for $H \in \mathcal{C}$, thus proving that the map in equation (\ref{eNotInjective}) is indeed not injective.

\end{document}